\documentclass[12pt, reqno]{amsart}
\usepackage{amsmath, amsthm, amscd, amsfonts, amssymb, graphicx, color, mathrsfs}
\usepackage[bookmarksnumbered, colorlinks, plainpages]{hyperref}
\hypersetup{colorlinks=true,linkcolor=red, anchorcolor=green, citecolor=cyan, urlcolor=red, filecolor=magenta, pdftoolbar=true}
\usepackage{hyperref}
\newtheorem{theorem}{Theorem}[section]
\newtheorem{lemma}[theorem]{Lemma}

\newtheorem{corollary}[theorem]{Corollary}
\theoremstyle{definition}

\theoremstyle{remark}
\newtheorem{remark}[theorem]{Remark}
\numberwithin{equation}{section}

\begin{document}

\setcounter{page}{1}

\title[The backward problem for a multi-term  \dots ]{The backward problem for a multi-term time-fractional diffusion equation
}

\author[R. Ashurov and  D. Shamuratov]{Ravshan Ashurov and Damir Shamuratov}

\address{\textcolor[rgb]{0.00,0.00,0.84}{Ravshan Ashurov:
V. I. Romanovskiy Institute of Mathematics Uzbekistan Academy of Sciences,
University street 9, Tashkent--100174,
Uzbekistan
}}
\email{\textcolor[rgb]{0.00,0.00,0.84}{ashurovr@gmail.com}}

\address{\textcolor[rgb]{0.00,0.00,0.84}{Damir Shamuratov:
V. I. Romanovskiy Institute of Mathematics Uzbekistan Academy of Sciences,
University street 9, Tashkent--100174,
Uzbekistan}}
\email{\textcolor[rgb]{0.00,0.00,0.84}{damirshamuratov5@gmail.com}}




\begin{abstract} This paper is devoted to the investigation of the backward problem for a multi-term time-fractional diffusion equation. Backward problems for fractional diffusion equations are typically studied using regularization methods due to their ill-posedness in the sense of Hadamard; that is, a small change in $u(T)$ may lead to large changes in the initial data. Nevertheless, we show that if sufficiently smooth current data are considered, then the solution exists, is unique, and is stable. A principal difficulty in the analysis of the backward problem stems from the structure of the solution, in which the multinomial Mittag--Leffler function appears in the denominator. Accordingly, a precise characterization of the asymptotic behavior of this function is required. Such asymptotic properties are nontrivial and have been rigorously established in the authors’ recent work, which serve as a fundamental basis for the present study. In addition, we investigate the conditional stability of the backward problem. It is shown that, although the problem is ill-posed in general, stability can be restored under an appropriate a priori bound imposed on the initial data. The main novelty of the paper lies in proving the best smoothing property of the
solution, showing that it belongs to the domain of the operator $A$ for any positive time.
\end{abstract} \maketitle

\section{Introduction}

The multi-term time fractional diffusion equation is a special type of fractional diffusion equation that can improve the efficiency of studying anomalous diffusion problems. This model can describe the diffusion phenomenon of a solute in a multi-scale medium. Such processes provide useful models for a wide range of non-homogeneous and non-stationary phenomena. For example, as shown by \cite{Hatano}, diffusion equations with time-fractional derivatives perform well in describing the long-tailed profile of a particle diffusing in a highly heterogeneous medium. Additionally, kinetic equations incorporating two fractional derivatives of different orders have been employed to describe subdiffusive motion in velocity fields. These multi-term fractional differential equations can accurately capture anomalous diffusion phenomena in complex systems and highly heterogeneous aquifers \cite{Schumer}. Therefore, the study of multi-term fractional diffusion equations is of great significance in physics and engineering.

Let $A$ be an arbitrary positive self-adjoint operator defined in a separable Hilbert space $H$ with domain $D(A)$. Denote the inner product in $H$ by $(\cdot,\cdot)$ and the norm by $\|\cdot\|$. Suppose further that the inverse operator $A^{-1}$ is compact. 
Then it is well known that $A$ possesses a complete orthonormal system of eigenfunctions $\{v_k\}$ and a countable set of positive eigenvalues $\{\lambda_k\}$ with a unique accumulation point at $+\infty$. 
We assume that the eigenvalues are ordered in a nondecreasing manner, i.e.,
\[
0 < \lambda_1 \leq \lambda_2 \leq \cdots \to +\infty.
\]

The Caputo fractional derivative of order $0<\rho_j<1$ $(j=1,\dots,M)$ for a function $h(t)$ is defined as (see \cite{Kilbas}, p. 90):
\[
\partial_t^{\rho_j} h(t) =
\frac{1}{\Gamma(1-\rho_j)} 
\int_0^t 
\frac{h'(s)}{(t-s)^{\rho_j}} ds, \quad t>0,
\]
provided that the right-hand side exists. Here, \(\Gamma(\cdot)\) is Euler’s gamma function.

In this paper, we consider the backward problem for the multi-term time-fractional diffusion equation with Caputo fractional derivatives.

Consider the Cauchy-type multi-term problem with inverse time:
\begin{equation}\label{1.1}
\begin{cases}
\displaystyle \sum_{j=1}^{M} q_j \, \partial_t^{\rho_j} u(t) + A u(t) = f(t), & 0 < t < T, \\[1mm]
u(T) = \Phi,
\end{cases}
\end{equation}
where \(0<\rho_M<\rho_{M-1}<\dots<\rho_1< 1\), \(q_j>0\) for \(j=1,\dots,M\) (\(q_1 = 1\) without loss of generality), \(\Phi \in H\), and \(f(t) \in C((0, T); H)\) are given vectors. This problem is called the \emph{backward problem}.

The standard formulation of the Cauchy problem for equation \eqref{1.1} has the form:
\begin{equation}\label{1.2}
\begin{cases}
\displaystyle \sum_{j=1}^{M} q_j \, \partial_t^{\rho_j} u(t) + A u(t) = f(t), & 0 < t < T, \\[1mm]
u(0) = \varphi,
\end{cases}
\end{equation}
where \(\varphi \in H\) is a given vector. This problem is called the \emph{forward problem}. To investigate the backward problem, one usually uses the properties of the solution to the forward problem.

The backward problem for diffusion processes plays a significant role in engineering, as it aims to recover the previous state of a physical field from its final data. Such problems have numerous practical applications. In particular, they characterize the blurring effect, while the corresponding backward formulations provide a mathematical framework for the deblurring process in image restoration. However, whether the Riemann–Liouville or Caputo fractional derivative is used, these problems remain \emph{ill-posed} in the sense of Hadamard. This means that even a very small change in the final data \(u(T)\) can cause large changes in the initial data $u(0)$. The backward problem of the classical diffusion equation is severely ill-posed (e.g., Isakov \cite{Isakov}, p. 21), and any Lipschitz-type estimate in the Sobolev norm is impossible.

Compared with the backward problem of classical heat equations (retrospective inverse problem, see \cite{Kabanikhin}, p. 214), studying the backward problem for fractional equations, especially the multi-term time-fractional diffusion equation is more challenging. For example, under the time reversal \(t' = T - t\), we have
\[
\partial_t = - \partial_{t'} \quad \text{while} \quad \partial_t^{\rho_j} \neq - \partial_{t'}^{\rho_j} \quad \text{for} \quad \rho_j \in (0, 1), \quad j=1,\dots,M,
\]
which prevents the direct application of standard techniques for the classical diffusion equation. Moreover, the appearance of the multinomial Mittag-Leffler function makes the study of forward and backward problems for fractional equations more complicated in both theoretical analysis and numerical calculations.

In the case of the Caputo time derivative, the backward problem for the single-term fractional diffusion equation for various elliptic differential operators $A$ has been considered by a number of authors. Let us mention only some of
these works. For the case of the second-order symmetric elliptic operator \(A\), Sakamoto and Yamamoto \cite{SakamotoYamamoto} established the unique existence of weak solutions and proved the stability of the backward problem in time, as well as the uniqueness in determining the initial value. The non-symmetric case was considered by Florida, Li, and Yamamoto \cite{Florida}.
In the case of the Riemann–Liouville time derivative, the backward problem for the single-term fractional diffusion equation was studied by Alimov and Ashurov \cite{AlimovAshurov}. 

Since backward problems for single-term or multi-term time-fractional diffusion equations are ill-posed, many authors have considered various regularization methods to determine the initial condition. In the single-term case, the authors in \cite{TingWei} formulated the backward problem as a variational problem using the Tikhonov regularization method and obtained an approximation to the minimizer of the variational problem by employing a conjugate gradient method. In \cite{Tuan}, the authors considered a backward problem for a time-space fractional diffusion equation with a nonlinear source. Under certain assumptions, they established the existence and uniqueness of local mild solutions to the nonlinear problem and proposed a regularization method to approximate the solution. Furthermore, the convergence rate of the regularized solution was established. In \cite{Liu}, the backward problem in the one-dimensional case is addressed using a quasi-reversibility regularization scheme, accompanied by a full theoretical analysis and numerical tests. In the multi-term case, the authors in \cite{JinWen} investigated a backward problem for the multi-term time–space fractional diffusion equation, which was ill-posed. Using the quasi-reversibility regularization method, they provided a regularized solution based on properties of the Fourier transform and Mittag-Leffler functions. In \cite{Sun}, the authors focused on the well-posedness and convergence analysis of the solution to the backward problem using the fractional-order quasi-reversibility method.

It is worth mentioning that the problem \eqref{1.1} is ill-posed in the sense of Hadamard due to the lack of stability of the solution. Nevertheless, we will show that if sufficiently smooth current information is considered, the solution exists and is unique. Our main contribution lies in proving the best possible smoothing property for the backward problem \eqref{1.1}, namely,
$ u(t) \in D(A)$ for any $t > 0$. In Theorem \ref{thm4.1}, it is shown that the smoothing in \(D(A)\) is the best possible, and the solution cannot be smoother than \(D(A)\) for \(t > 0\) if \(\varphi \in H\). To prove this, we require the asymptotic behavior of the multinomial Mittag-Leffler function, since it appears in the denominator of the solution. As shown in \cite{AshurovShamuratov}, the asymptotic behavior of the multinomial Mittag-Leffler function appearing in our solution is provided and proved.

The remainder of this paper is organized as follows. In Section \ref{sec:2}, some preliminaries are presented. In Section \ref{sec:3}, the forward problem for a multi-term time-fractional differential equation is studied. Under appropriate assumptions on the initial data and the source term, the existence and uniqueness of the solution are established, and derive an a priori estimate that is essential for the analysis of the backward problem. In Section \ref{sec:4}, we investigate the backward problem. First, we show that the problem is ill-posed in the sense of Hadamard. Then, under suitable assumptions, we establish existence and uniqueness results and derive corresponding stability estimates. Moreover, we obtain a conditional stability result under an appropriate a priori bound on the initial data. Finally, Section \ref{sec5} contains the conclusion.
\section{Preliminaries}
\label{sec:2}
In this section, we recall several lemmas and auxiliary results that will be used throughout the paper.

Let $\varepsilon$ be an arbitrary real number. We introduce the fractional power of the operator $A$ acting in $H$ according to the following rule:
$$
\begin{cases}
A^{\varepsilon} g = \displaystyle \sum_{k=1}^{\infty} \lambda_k^{\varepsilon} g_k\,v_k, \\[2mm]
D(A^{\varepsilon}) = \left\{ g \in H : \displaystyle \sum_{k=1}^{\infty} \lambda_k^{2\varepsilon} |g_k|^2 < \infty \right\}, \\[2mm]
\|A^{\varepsilon} g\|^2 =  \displaystyle \sum_{k=1}^{\infty} \lambda_k^{2\varepsilon} |g_k|^2 =\|g\|_{\varepsilon}^2,
\end{cases}
$$
where $g_k = (g, v_k)$ are the Fourier coefficients of a function $g \in H$.

The solution of problem \eqref{1.1} involves the multivariable (multinomial) Mittag--Leffler function. Therefore, we first present the definition of the multinomial Mittag--Leffler function.

The multinomial Mittag–Leffler function is defined as  \cite{MasahiroYamamoto}

\[
E_{(\beta_1, \dots, \beta_M),\, \beta_0}(z_1, \dots, z_M) =
\sum_{k=0}^{\infty} 
\sum_{k_1 + \dots + k_M = k} 
\binom{k}{k_1, \dots, k_M} 
\prod_{j=1}^M z_j^{k_j} 
\frac{1}{\Gamma(\beta_0 + \sum_{j=1}^M \beta_j k_j)},
\]
where we assume \( \beta_0 >0\), \(0 < \beta_j < 1\), \(z_j \in \mathbb{C}\) (\(j = 1, \dots, M\)), and \(\binom{k}{k_1, \dots, k_M}\) denotes
the multinomial coefficient
\[
\binom{k}{k_1, \dots, k_M} = \frac{k!}{k_1! \cdots k_M!},\quad
\text{with} \quad
k = \sum_{j=1}^{M} k_j,
\]
where \(k_j\), \(1 \le j \le M\), are non-negative integers.

To derive an upper bound for the multinomial Mittag–Leffler function and to prove the convergence of the solution, we recall the following lemmas.
\begin{lemma}\label{l2.1} \cite{MasahiroYamamoto} Let $0<\beta<2$ and
$1>\rho_1>\cdots>\rho_M>0$ be given. Assume that $\frac{\rho_1\pi}{2}<\mu<\rho_1\pi,\,\,
\mu \le |\arg(z_1)| \le \pi,
$
and there exists a constant $K>0$ such that
$
-K \le z_j < 0, j=2,\dots,M.
$
Then there exists a constant $C>0$, depending only on $\mu$, $K$, $\rho_j$ $(j=1,\dots,M)$
and $\beta$, such that
\[
\left|
E_{(\rho_1,\rho_1-\rho_2,\dots,\rho_1-\rho_M),\beta}
(z_1,\dots,z_M)
\right|
\le
\frac{C}{1+|z_1|}.
\]
\end{lemma}
For later use, we adopt the abbreviation
$$
E_{\rho',\rho_1}(-\lambda_k t^{\rho_1},*)=E_{(\rho_1,\rho_1-\rho_2,...,\rho_1-\rho_M),\rho_1}(-\lambda_k t^{\rho_1},-q_2t^{\rho_1-\rho_2},...,-q_Mt^{\rho_1-\rho_M}),
$$
$$
\rho'=(\rho_1,\rho_1-\rho_2,...,\rho_1-\rho_M).
$$

We will also use a coarser estimate with positive eigenvalues $\lambda_k$ and $0<\varepsilon<1$:
\begin{equation}\label{2.1}
\left| t^{\rho_1-1} E_{\rho',\rho_1}(-\lambda_k t^{\rho_1},*) \right|
\le
\frac{C t^{\rho_1-1}}{1+\lambda_k t^{\rho_1}}
\le
C \lambda_k^{\varepsilon-1} t^{\varepsilon \rho_1-1}, 
\qquad t>0,
\end{equation}
which is easy to verify. Indeed, let $t^{\rho_1}\lambda_k<1$, then $t<\lambda_k^{-1/\rho_1}$ and
\[
t^{\rho_1-1}
=
t^{\rho_1-\varepsilon\rho_1} t^{\varepsilon\rho_1-1}
<
\lambda_k^{\varepsilon-1} t^{\varepsilon\rho_1-1}.
\]

If $t^{\rho_1}\lambda_k \ge 1$, then $\lambda_k^{-1} \le t^{\rho_1}$ and
\[
\lambda_k^{-1} t^{-1}
=
\lambda_k^{-1+\varepsilon}\lambda_k^{-\varepsilon} t^{-1}
\le
\lambda_k^{\varepsilon-1} t^{\varepsilon\rho_1-1}.
\]
\begin{lemma}\label{l2.4}\cite{AshurovShamuratov}
    Let $ \beta > 2\rho_1$ and $1 > \rho_1 > \cdots > \rho_M > 0$ be given. 
Assume that $\rho_1 \pi/2 < \mu < \rho_1 \pi$, 
and there exists $K > 0$ such that $-K \le z_j < 0 \; (j = 2, \ldots, M)$. Then, we have the following asymptotic formulas in which $p$ is an arbitrary positive integer:
$$
E_{{\rho}',\beta}(z_1, \ldots, z_M)=-\frac{1}{z_1}\frac{1}{\Gamma(\beta-\rho_1)}-\frac{1}{z_1^2}\left(\frac{1}{\Gamma(\beta-2\rho_1)}-
\sum_{j=2}^{M}
\frac{z_j}{\Gamma(\beta-\rho_1-\rho_j)}\right)$$
$$
-\sum_{k=3}^p \frac{C_k(z_2,..,z_M,\rho_j,\beta)}{z_1^k}+O\left(|z_1|^{-p-1}\right), \quad  |z_1|\to\infty , \quad  \mu \le |\arg(z_1)| \le \pi.
$$
\end{lemma}
The proof follows from the integral representation of the multinomial Mittag–Leffler function given in \cite{MasahiroYamamoto}. Using this representation and proceeding similarly to the derivation of the classical asymptotic properties of the Mittag–Leffler function (see \cite{Gorenflo}), we obtain the desired result. For more details, see \cite{AshurovShamuratov}.

\section{Asymptotic behavior of the multinomial Mittag-Leffler function}
\label{sec:3}
The investigation of the backward problem is usually based on the properties of the solution to the forward problem.
In this section, the forward problem for a multi-term time-fractional differential equation is studied, and an estimate \eqref{thm3.1}, which is necessary for the backward problem, is also proved.

\begin{theorem}\label{thm3.1} Let $\varphi \in H$ and  $f(t) \in C([0,T]; D(A^{\varepsilon}))$ for some $\varepsilon \in(0,1)$. 
Then, the forward problem \eqref{1.2} has a unique solution.
Moreover, there exists a constant \(C\), depending on \(\varepsilon\) and $\rho_1$, such that
\begin{equation}\label{t3.1}
\|u(t)\|_1\le C \Biggl( \|\varphi\| \sum_{j=1}^{M} t^{- \rho_j} + \max_{t \in [0,T]} \|  f(t) \|_{\varepsilon} \Biggr), \quad t>0.
\end{equation}
\end{theorem}
\begin{proof}
    Let us introduce the following formal series
\begin{equation}\label{3.1}
    u(t)=\sum_{k=1}^{\infty}\left[\varphi_k\left(1 - \lambda_k t^{\rho_1} E_{\rho',\rho_1+1}(-\lambda_k t^{\rho_1},*)\right)+\int_0^tf_k(t-\xi)\xi^{\rho_1-1}E_{\rho',\rho_1}(-\lambda_k \xi^{\rho_1},*)d\xi\right]v_k,
\end{equation}
where $\varphi_k$ and $f_k(t)$ are the Fourier coefficients of $\varphi$ and $f(t)$, respectively.
One can easily verify that the function \eqref{3.1} formally satisfies the conditions of problem \eqref{1.2} (see \cite{MasahiroYamamoto}).
In order to prove that the function defined by \eqref{3.1} is indeed a solution to the problem, it remains to justify this formal argument, namely, to show that the operators $A$ and $\partial_t^{\rho_j}$ can be applied term by term to the series \eqref{3.1}.
For convenience, we decompose the series \eqref{3.1} into two sums $\left(u(t)=u^1(t)+u^2(t)\right)$:
\begin{equation}\label{sum1}
    u^1(t)=\sum_{k=1}^{\infty}\varphi_k\left(1 - \lambda_k t^{\rho_1} E_{\rho',\rho_1+1}(-\lambda_k t^{\rho_1},*)\right)v_k,
\end{equation}
\begin{equation}\label{sum2}
    u^2(t)=\sum_{k=1}^{\infty}\left[\int_0^tf_k(t-\xi)\xi^{\rho_1-1}E_{\rho',\rho_1}(-\lambda_k \xi^{\rho_1},*)d\xi \right]v_k.
\end{equation}

\emph{Existence}.
Let $u_n(t)=u^1_n(t)+u^2_n(t)$ be the partial sum of series \eqref{3.1}. Then
\[
 Au^1_n 
=
\sum_{k=1}^{n}\lambda_k\varphi_k\left(1 - \lambda_k t^{\rho_1} E_{\rho',\rho_1+1}(-\lambda_k t^{\rho_1},*)\right)v_k.
\]

Due to the Parseval equality we may write
\[
\| A u^1_n\|^2
=
 \sum_{k=1}^{n}
\left|
\lambda_k\varphi_k\left(1 - \lambda_k t^{\rho_1} E_{\rho',\rho_1+1}(-\lambda_k t^{\rho_1},*)\right)
\right|^2.
\]

The following estimate holds for $t>0$, \cite{MasahiroYamamoto}
\[
\left|1 - \lambda_k t^{\rho_1} E_{\rho',\rho_1+1}(-\lambda_k t^{\rho_1},*)
\right|\le
C \sum_{j=1}^{M} \frac{t^{\rho_1 - \rho_j}}{1 + \lambda_k t^{\rho_1}}.
\]

It follows that, we have
\[
\| A u^1_n\|^2\le C^2\sum_{k=1}^{n}
\left|
\varphi_k\sum_{j=1}^{M} \frac{\lambda_kt^{\rho_1}}{1 + \lambda_k t^{\rho_1}}t^{- \rho_j}
\right|^2.
\]

We use the fact 
\[
\frac{\lambda_k t^{\rho_1}}{1 + \lambda_k t^{\rho_1}}
\le 1.
\]

Therefore,  we get
\[
\| A u^1_n\|^2\le C^2\left|\sum_{j=1}^{M} t^{- \rho_j}\right|^2\sum_{k=1}^{n}
\left|
\varphi_k
\right|^2 <\infty,\quad  t>0.\]

By inequality \eqref{2.1}, for $0<\varepsilon<1$, we obtain
\[\| A u^2_n\|^2=
\sum_{k=1}^{n} \lambda_k^{2}
\left|
\int_{0}^{t}f_k(t-\xi)
\xi^{\rho_1-1}
E_{\rho',\rho_1}(-\lambda_k \xi^{\rho_1},*)
\, d\xi
\right|^{2}
\le
C^2
\sum_{k=1}^{n}
\left(
\int_{0}^{t}
\xi^{\varepsilon\rho_1-1}
\lambda_k^{\varepsilon}
|f_k(t-\xi)|\, d\xi
\right)^{2}.
\]

Using the generalized Minkowski inequality,  we have
$$
C^2
\sum_{k=1}^{n}
\left(
\int_{0}^{t}
\xi^{\varepsilon\rho_1-1}
\lambda_k^{\varepsilon}
|f_k(t-\xi)|\, d\xi
\right)^{2}
\le
C^2\left(
\int_{0}^{t}
\xi^{\varepsilon\rho_1-1}
\left(
\sum_{k=1}^{n}
\left|
\lambda_k^{\varepsilon} f_k(t-\xi)
\right|^{2}
\right)^{1/2}
d\xi
\right)^{2}$$
$$\le \frac{C^2t^{2\varepsilon\rho_1}}{(\varepsilon\rho_1)^2} \, \max_{t \in [0,T]} \|  f(t) \|_{\varepsilon}^2.
$$

Therefore
\begin{equation}\label{estimate}
\| A u_n\|^2\le \left(C^2\|\varphi\|^2\left|\sum_{j=1}^{M} t^{- \rho_j}\right|^2+C^2_{\varepsilon,\rho_1} \, \max_{t \in [0,T]} \| f(t) \|_{\varepsilon}^2\right).
\end{equation}

Further, from equation \eqref{1.2} one has
\[
\sum_{j=1}^Mq_j\partial_t^{\rho_j}u_n(t) = -Au_n(t) + \sum_{k=1}^{n} f_k(t) v_k.
\]

Therefore, from the above reasoning, we obtain
\[
\left\|
\sum_{j=1}^{M} q_j \partial_t^{\rho_j} u_n(t)
\right\|^2
\le
\left(
C^2\|\varphi\|^2\left|\sum_{j=1}^{M} t^{- \rho_j}\right|^2 +C^2_{\varepsilon,\rho_1} \max_{t \in [0,T]} \|  f(t) \|_{\varepsilon}^2
\right)
+
\|f(t)\|^2.
\]

Thus, we have completed the rationale that \eqref{3.1} is a solution to the forward problem. 
It follows from \eqref{estimate} that estimate \eqref{t3.1} holds.

 \emph{Uniqueness}.    Suppose we have two solutions: $u_1(t),\, u_2(t)$ and set $u(t) = u_1(t)-u_2(t).$ Then, we have
\begin{equation}\label{3.3}
\begin{cases}
\displaystyle \sum_{j=1}^{M} q_j \, \partial_t^{\rho_j} u(t) + A u(t) = 0, \\[2mm]
u(0) = 0.
\end{cases}
\end{equation}
Set
\[
w_k(t) = (u(t), v_k).
\]
It follows from \eqref{3.3} that
\[
\sum_{j=1}^Mq_j\partial_t^{\rho_j} w_k(t)
= \left(\sum_{j=1}^Mq_j\partial_t^{\rho_j} u(t), v_k\right)
= - (A u(t), v_k)
= - (u(t), A v_k)
= - \lambda_k w_k(t).
\]
Consequently, $w_k(t)$ satisfies the Cauchy-type problem
\begin{equation}\label{3.4}
\begin{cases}
\displaystyle \sum_{j=1}^{M} q_j \, \partial_t^{\rho_j} w_k(t) + \lambda_k w_k(t) = 0, \\[2mm]
w_k(0) = 0.
\end{cases}
\end{equation}
From \eqref{3.1} ($\varphi_k=0, \,f_k=0$), it follows that the function defined by equation \eqref{3.4} is identically zero:
$
w_k(t) \equiv 0.
$
Consequently, due to the completeness of the system of eigenfunctions \(\{v_k\}\),
we have $ u(t) = 0 $.
Thus, the uniqueness of the solution of problem \eqref{1.2} is established.

Thus the proof of Theorem \ref{thm3.1} is complete.
\end{proof}
Here is an obvious consequence of estimate \eqref{t3.1}:
\begin{corollary}\label{cor3.2}
    Let $\varphi \in H$ and $f(t) \in C([0,T];D(A^{\varepsilon}))$ for some $\varepsilon \in(0,1)$.
Then there exists a constant $C$ depending on $T,\,\rho_1$ and $\varepsilon$ such that
\begin{equation}\label{3.7}
\|u(T)\|_{1} \le C\Big(\|\varphi\|+\max_{t\in[0,T]}\|f(t)\|_{\varepsilon}\Big).
\end{equation}
\end{corollary}

\section{Lower bound for an infinitesimal denominator}
\label{sec:4}

In this section, the existence and uniqueness of the solution to the backward problem are established, and stability estimates of the solution are derived.

Problem \eqref{1.1} is ill-posed in the sense of Hadamard, since even a small change in $u(T)$ measured in the norm of $H$ can result in arbitrarily large changes in the initial data. To illustrate this, assume that $f(t) \equiv 0$ and
\[
u(T) = \lambda_k^{-1+\varepsilon} v_k, \quad \varepsilon > 0.
\]
Then the corresponding solution of problem \eqref{1.1} is given by
\[
u(t) = \lambda_k^{-1+\varepsilon}  
\frac{1 - \lambda_k t^{\rho_1} E_{\rho',\rho_1+1}(-\lambda_k t^{\rho_1},*)}
{1 - \lambda_k T^{\rho_1} E_{\rho',\rho_1+1}(-\lambda_k T^{\rho_1},*)}
 v_k,
\]
and
\[
 u(0)
=
\lambda_k^{-1+\varepsilon}  
\frac{1}{1 - \lambda_k T^{\rho_1} E_{\rho',\rho_1+1}(-\lambda_k T^{\rho_1},*)}
 v_k.
\]

Therefore, on the one hand,
\[
\|u(T)\| = \lambda_k^{-1+\varepsilon}
\]
and it tends to zero as $k \to \infty$ (even $\|u(T)\|_a \to 0$ for any $a < 1-\varepsilon$),
on the other hand, according to the asymptotic estimate given in Lemma \ref{l2.4},
\[
\left\| u(0) \right\|
=
\lambda_k^{-1+\varepsilon} 
\frac{1}{1 - \lambda_k T^{\rho_1} E_{\rho',\rho_1+1}(-\lambda_k T^{\rho_1},*)}
\to \infty \quad \text{when } k \to \infty.
\]

However, if the norm of $u(T)$ is considered in the space $D(A)$, the situation changes significantly. In this case, we have
\[
\|u(T)\|_1 = \lambda_k^{\varepsilon},
\]
which diverges as $k \to \infty$.
\begin{theorem}\label{thm4.1}
    Let $f(t) \equiv 0$. Then for any $\Phi \in D(A)$ problem \eqref{1.1} has a unique solution. 
Moreover, there exist constants $C_{1}, C_{2} > 0$ such that
\begin{equation}\label{4.1}
C_{1}\,\|u(0)\|
\le
\|u(T)\|_1
\le
C_{2}\,\|u(0)\|.
\end{equation}
\end{theorem}
\begin{proof}

Let $\Phi \in D(A)$ and let $\Phi_k$ be its Fourier coefficients. Then
\[
\|\Phi\|_{1}^2 =
\sum_{k=1}^{\infty} \lambda_k^{2} |\Phi_k|^{2} < \infty.
\]
From the condition in \eqref{1.1}, and the fact that $f(t)\equiv 0$ in \eqref{3.1}, we obtain
\[
\varphi_k =
\frac{\Phi_k}{1 - \lambda_k T^{\rho_1} E_{\rho',\rho_1+1}(-\lambda_k T^{\rho_1},*)}.
\]
By Lemma \ref{l2.4}, we have 
\begin{equation}\label{a}
\begin{aligned}
\sum_{k=1}^{\infty} \varphi_k^{2}
&=
\sum_{k=1}^{\infty}
\frac{\Phi_k^{2}}{\left(1 - \lambda_k T^{\rho_1} E_{\rho',\rho_1+1}(-\lambda_k T^{\rho_1},*)\right)^{2}}\\
&=\sum_{k=1}^{\infty}
\frac{\Phi_k^{2}}{\left(1 - \lambda_k T^{\rho_1} \left(\frac{1}{\lambda_k T^{\rho_1}}-\frac{C}{(\lambda_k T^{\rho_1})^2}+O((\lambda_k T^{\rho_1})^{-3})\right)\right)^{2}}
\\&=
 \frac{1}{C^2}T^{2\rho_1}\sum_{k=1}^{\infty}
\lambda_k^{2}\Phi_k^{2}\left(\frac{1}{1+O((\lambda_k T^{\rho_1})^{-3})}\right)^2\le \frac{1}{C^2}T^{2\rho_1}\sum_{k=1}^{\infty}
\lambda_k^{2}\Phi_k^{2}
<
\infty.
\end{aligned}
\end{equation}
 The following function (see \eqref{3.1})
\[
u(t)=
\sum_{k=1}^{\infty}\varphi_k
\left(1 - \lambda_k t^{\rho_1} E_{\rho',\rho_1+1}(-\lambda_k t^{\rho_1},*)\right) v_k,
\]
is the unique solution to the forward problem \eqref{1.2} with $f(t)\equiv 0$ and the initial function $\varphi$.  
Moreover,
$
u(T)=\Phi
$
and from \eqref{a}, we have
\[
\|u(0)\|=\|\varphi\|
\le
C\|\Phi\|_{1}
=
C\|u(T)\|_{1}.
\]
The second inequality in \eqref{4.1} is already proved in Theorem \ref{thm3.1} (Corollary \ref{cor3.2}).  
Thus, Theorem \ref{thm4.1} is proved.
\end{proof}
\begin{remark}
It follows from the Theorem that in the estimate \eqref{4.1} one cannot replace 
$\|u(T)\|_{1}$ by $\|u(T)\|_{1-\varepsilon}$ with $\varepsilon > 0$. A small change of $u(T)$ in the norm of space $D(A)$ leads to small changes in the initial data.
\end{remark}
\begin{theorem}\label{thm4.2}
    Let 
$
f(t) \in C([0, T]; D(A^{\varepsilon}))
$
with some \(\varepsilon \in (0,1)\). Then for any 
\(\Phi \in D(A)\), problem \eqref{1.1} has a unique solution.
Moreover, there exists a constant $C>0$ such that
\begin{equation}
\|u(0)\|
\le
C\left(
\|u(T)\|_{1}
+
\max_{t \in [0,T]}
\| f(t)\|_{\varepsilon}
\right).
\end{equation}
\end{theorem}
\begin{proof}
Consider the following two auxiliary problems $(u=v+w)$ :
\begin{equation}\label{4.4}
\begin{cases}
\displaystyle \sum_{j=1}^{M} q_j \partial_t^{\rho_j} v(t) + A v(t) = f(t), & 0 < t < T,\\
 v(0) = 0,
\end{cases}
\end{equation}
and
\begin{equation}\label{4.5}
\begin{cases}
\displaystyle \sum_{j=1}^{M} q_j \partial_t^{\rho_j} w(t) + A w(t) = 0, & 0 < t < T,\\
w(T) = \Phi - v(T).
\end{cases}
\end{equation}
If $ f(t) \in C([0, T ]; D(A^{\varepsilon}))$, then there exists a unique solution to problem \eqref{4.4} and (see Corollary \ref{cor3.2})
\begin{equation}\label{4.6}
\|v(T)\|_{1} \le C \max_{t \in [0,T]} \|f(t)\|_{\varepsilon}.
\end{equation}
If $\Phi \in D(A)$, then there exists a unique solution to problem \eqref{4.5} and (see Theorem \ref{thm4.1})
\begin{equation}\label{4.7}
\| w(0)\|\le C \|w(T)\|_{1}.
\end{equation}
Setting $u = v + w$, we see that
\[
u(T) = \Phi - v(T) + v(T) = \Phi.
\]
Then we can verify that $u$ is the unique solution to problem \eqref{1.1} and the estimates \eqref{4.6} and \eqref{4.7} imply
\[
\| u(0)\|
=
 \| w(0)\|
\le
C\|w(T)\|_{1}
\le
C(\|\Phi\|_{1} + \|v(T)\|_{1})
\le
C\left(
\|\Phi\|_{1}
+
\max_{t \in [0,T]}
\| f(t)\|_{\varepsilon}
\right).
\]
Thus, Theorem \ref{thm4.2} is proved.
\end{proof}
\subsection{Conditional stability}

Suppose that the initial data $u(0) = \varphi$ satisfies the following a priori estimate:
\begin{equation}\label{4.8}
\|\varphi\|_{\varepsilon}^{2} 
= \sum_{k=1}^{\infty} \lambda_k^{2\varepsilon} |\varphi_k|^{2} 
\leq B_0^{2},
\end{equation}
where $\varepsilon > 0$ and $B_0$ are positive constants, and consider a class of functions that satisfy this
condition.

\emph{Conditional stability} refers to a situation where a problem that is generally unstable becomes stable when restricted to a certain class of functions. As noted above, the solution to problem \(\eqref{1.1}\) is inherently unstable: even a small change in \(u(T)\) can produce a large change in the initial data \(u(0)\). However, if stability is ensured for a specific subset of initial data \(u(0)\) (see \(\eqref{4.8}\)), the problem is said to be conditionally stable.

The following statement is true:

\begin{theorem}\label{thm4.3}
    Let $\varphi\in D(A^{\varepsilon})$ satisfy condition \eqref{4.8}. Then there is a constant $C$ depending on $T,\,\varepsilon$ and $\rho_1$ such that
 $$
\|\varphi\|\le C\left[\|\Phi\|+\max_{t \in [0,T]} \|  f(t) \|\right]^{\frac{\varepsilon}{1+\varepsilon}}B_0^{\frac{1}{1+\varepsilon}}.
$$
\end{theorem}
\begin{proof}
    Let us take a solution \eqref{3.1} with an unknown initial function $\varphi$ and use condition
$u(T) = \Phi$ to determine this unknown function. Then the Fourier coefficients of $\varphi$ has
the form
$$
\varphi_k=\frac{1}{1 - \lambda_k T^{\rho_1} E_{\rho',\rho_1+1}(-\lambda_k T^{\rho_1},*)}\left(\Phi_k-\int_0^Tf_k(T-\xi)\xi^{\rho_1-1}E_{\rho',\rho_1}(-\lambda_k \xi^{\rho_1},*)d\xi\right).
$$
We introduce the following notation:
$$B_k=\Phi_k-\int_0^Tf_k(T-\xi)\xi^{\rho_1-1}E_{\rho',\rho_1}(-\lambda_k \xi^{\rho_1},*)d\xi.$$
Then, for $2=\cfrac{2\varepsilon}{1+\varepsilon}+\cfrac{2}{1+\varepsilon}$, we have 
$$
\|\varphi\|^{2} 
= \sum_{k=1}^{\infty} \left| \frac{B_k}{1 - \lambda_k T^{\rho_1} E_{\rho',\rho_1+1}(-\lambda_k T^{\rho_1},*)} \right|^{2}
= \sum_{k=1}^{\infty} \frac{|B_k|^{\frac{2\varepsilon}{1+\varepsilon}} \, |B_k|^{\frac{2}{1+\varepsilon}}}{\left| 1 - \lambda_k T^{\rho_1} E_{\rho',\rho_1+1}(-\lambda_k T^{\rho_1},*) \right|^{2}}.
$$
Applying the Hölder inequality with parameters 
$p = \cfrac{1+\varepsilon}{\varepsilon}$ and $q = 1+\varepsilon$, we obtain
$$
\|\varphi\|^2 \le 
\left( \sum_{k=1}^{\infty} |B_k|^2 \right)^{\frac{\varepsilon}{1+\varepsilon}}$$
$$\times\left( \sum_{k=1}^{\infty} 
\frac{1}{|1 - \lambda_k T^{\rho_1} E_{\rho',\rho_1+1}(-\lambda_k T^{\rho_1},*)|^{2\varepsilon}}
\left| \frac{B_k}{1 - \lambda_k T^{\rho_1} E_{\rho',\rho_1+1}(-\lambda_k T^{\rho_1},*)} \right|^2
\right)^{\frac{1}{1+\varepsilon}}.
$$
We use Lemma \ref{l2.4} to get
$$
\sum_{k=1}^{\infty}
\frac{1}{|1 - \lambda_k T^{\rho_1} E_{\rho',\rho_1+1}(-\lambda_k T^{\rho_1},*)|^{2\varepsilon}}
\left| \frac{B_k}{1 - \lambda_k T^{\rho_1} E_{\rho',\rho_1+1}(-\lambda_k T^{\rho_1},*)} \right|^2
$$
$$\le
\frac{T^{2\varepsilon\rho_1}}{C^{2\varepsilon}}
\sum_{k=1}^{\infty}
\lambda_k^{2\varepsilon} |\varphi_k|^2
\le
\frac{T^{2\varepsilon\rho_1}B_0^2}{C^{2\varepsilon}}.
$$
Using the inequality $(a-b)^2\le 2(a^2+b^2)$, we have
$$
\sum_{k=1}^{\infty} |B_k|^2 
\le 
2\|\Phi\|^2 
+ 
2 \sum_{k=1}^{\infty}
\left[
\int_0^T\left|f_k(T-\xi)\xi^{\rho_1-1}E_{\rho',\rho_1}(-\lambda_k \xi^{\rho_1},*)d\xi\right|
\right]^2.
$$
By inequality \eqref{2.1}, for $0<\varepsilon<1$, and using the generalized Minkowski inequality, we obtain
$$
\sum_{k=1}^{\infty}
\left[
\int_0^T\left|f_k(T-\xi)\xi^{\rho_1-1}E_{\rho',\rho_1}(-\lambda_k \xi^{\rho_1},*)d\xi\right|
\right]^2\le C^2\sum_{k=1}^{n}
\left[
\int_{0}^{T}
\xi^{\varepsilon\rho_1-1}
\lambda_k^{\varepsilon-1}
|f_k(T-\xi)|\, d\xi
\right]^{2}
$$
$$\le
C^2\left(
\int_{0}^{T}
\xi^{\varepsilon\rho_1-1}
\left(
\sum_{k=1}^{n}
\left|
\lambda_k^{\varepsilon-1} f_k(T-\xi)
\right|^{2}
\right)^{1/2}
d\xi
\right)^{2}$$
$$\le C^2\left(
\int_{0}^{T}
\xi^{\varepsilon\rho_1-1}
\left(
\sum_{k=1}^{n}
\left|
 f_k(T-\xi)
\right|^{2}
\right)^{1/2}
d\xi
\right)^{2} \le  \frac{C^2T^{2\varepsilon\rho_1}}{(\varepsilon\rho_1)^2} \, \max_{t \in [0,T]} \|  f(t) \|^2.
$$
Therefore
$$
\sum_{k=1}^{\infty} |B_k|^2 
\le 
2\|\Phi\|^2 
+ 
2 \frac{C^2T^{2\varepsilon\rho_1}}{(\varepsilon\rho_1)^2} \, \max_{t \in [0,T]} \|  f(t) \|^2.
$$
Finally, we get
$$
\|\varphi\|\le C\left[\|\Phi\|+\max_{t \in [0,T]} \|  f(t) \|\right]^{\frac{\varepsilon}{1+\varepsilon}}B_0^{\frac{1}{1+\varepsilon}}.
$$
Thus, Theorem \ref{thm4.3} is proved.

\end{proof}

\section{Conclusion}\label{sec5}

In this study, we investigated the backward problem for a multi-term time-fractional diffusion equation with Caputo fractional derivatives in a Hilbert space setting. The analysis commenced with the corresponding forward problem, where existence, uniqueness, and essential a priori estimates of the solution were established, providing a solid foundation for the backward problem.
The backward problem, being ill-posed in the sense of Hadamard, presents significant analytical challenges, particularly due to the appearance of the multinomial Mittag--Leffler function in the denominator of the solution. A precise characterization of the asymptotic behavior of this function was necessary for the analysis. These asymptotic properties, which are nontrivial, have been rigorously established in the authors' recent works \cite{AshurovShamuratov}, forming a crucial basis for the present study.
Under suitable smoothness assumptions on the final data, we proved the existence and uniqueness of the solution and derived corresponding stability estimates. Furthermore, we demonstrated the best possible smoothing property of the solution, showing that it belongs to the domain of the operator \(A\) for any positive time. The analysis, based on spectral methods and asymptotic estimates of the multinomial Mittag--Leffler function, not only addresses the ill-posedness but also provides deeper insight into the structure and regularity of the solution. In addition, we investigated the conditional stability of the backward problem. Our results show that, although the problem is generally ill-posed, stability can be achieved under a suitable a priori bound imposed on the initial data.
Overall, the results contribute to a better theoretical understanding of backward multi-term time-fractional diffusion problems and provide a rigorous foundation for further analytical investigations in this area. In future work, the authors will focus on the development of efficient regularization methods, including generalized Tikhonov as well as other regularization methods, for the stable reconstruction of the solution.

\end{document}